\documentclass{amsart}
\usepackage{amsmath,amsthm,amsfonts,url,tikz}
\usetikzlibrary{matrix, calc, arrows}

\newcommand{\inv}[1]{{#1}^{\vee}}
\newcommand{\g}[3]{g(#1,#2,#3)}
\newcommand{\kro}[3]{\g{#1}{#2}{#3}}
\newcommand{\pleth}[3]{a_{#1,#2}^{#3}}

\newcommand{\CC}{\mathbb{C}}
\newcommand{\ZZ}{\mathbb{Z}}

\newcommand{\boxit}[3]{\Box_{#1,#2}(#3)}
\renewcommand{\P}[1]{\mathcal{P}(#1)}
\newcommand{\PP}[1]{\mathcal{P}^+(#1)}
\newcommand{\Sym}{\textit{Sym}}
\newcommand{\tab}[2]{\mathcal{T}_{#1}(#2)}

\newtheorem{theorem}{Theorem}

\newtheorem{proposition}{Proposition}
\newtheorem{corollary}{Corollary}
\newtheorem{lemma}{Lemma}
{\theoremstyle{definition}
\newtheorem*{remark}{Remark}
\newtheorem{example}{Example}
}

\title[Rectangular Symmetries]{Rectangular Symmetries for coefficients of symmetric functions}
\date{\today}

\author{Emmanuel Briand}
\address{
Emmanuel Briand,
Departamento de  Matem\'atica Aplicada I,
Escuela T\'ecnica Superior de Ingenier\'{\i}a Inform\'atica,
Universidad de Sevilla,
Avda. Reina Mercedes, 
41012 Sevilla,
Spain.}
\email{ebriand@us.es}

\author{Rosa Orellana}
\address{Rosa Orellana, Dartmouth College, Mathematics Department, 6188 Kemeny Hall, Hanover, NH 03755, USA.}
\email{rosa.c.orellana@dartmouth.edu}

\author{Mercedes Rosas}
\address{
 Mercedes Rosas,
Departamento de \'Algebra,
Facultad De Matem\'aticas,
Universidad de Sevilla,
Avda. Reina Mercedes, 
41012 Sevilla,
Spain.}
\email{mrosas@us.es}

\begin{document}

\begin{abstract} We show that some of the main structural constants for symmetric functions (Littlewood-Richardson coefficients, Kronecker coefficients, plethysm coefficients, and the Kostka--Foulkes polynomials) share symmetries related to the operations of taking complements with respect to rectangles and adding rectangles. \end{abstract}

\maketitle

\section{Introduction}
Four families of coefficients of great importance in the theory of symmetric functions are: the Kostka numbers (and their deformations, the Kostka--Foulkes polynomials), the Littlewood-Richardson coefficients, the Kronecker coefficients, and the plethysm coefficients.  The importance of these families of numbers come from their applicability to many different fields of mathematics such as representation theory, invariant theory and algebraic geometry, as well as physics and computer science.

The Littlewood-Richardson and Kronecker coefficients satisfy many well-known symmetries involving permutation of indices and conjugation.
 
These symmetries often lead to better understanding of the objects they enumerate, to simplifications in the number of cases in proofs, and in some cases, can be used to simplify computations.   In this article we present symmetries for Littlewood-Richardson, Kronecker, and plethysm coefficients, and for the Kostka-Foulkes polynomials, that involve the operations of (i)  taking complements in rectangles, or (ii) adding ``tall"  rectangles to the parts. The symmetries of type (i) actually follow from duality between representations of general linear groups, and those of type (ii) from factoring by determinant representations. We give  an elementary approach using the language of symmetric polynomials (instead of representations). In this language, these symmetries appear as evaluation at the inverses of the variables, for type (i), and factorization by the product of the variables, for type (ii).

In more detail, let  $\boxit{k}{a}{\lambda}$ be the complement of $\lambda$ with respect a $k \times a$ rectangle, as illustrated in Figure \ref{figure:complement}.
\begin{figure}[ht]
\begin{tikzpicture}[scale=.3]
\draw[thick] (2,0)--(2,6)--(5,6)--(5,5)--(6,5)--(6,3.5)--(7.5,3.5)--(7.5,2.5)--(9.5,2.5)--(9.5,0)--cycle;
\draw (4.5,2.5) node {\small $\boxit{k}{a}{\lambda} $};
\draw[dashed] (5,6)--(9.5,6)--(9.5,2.5);
\draw (8,4.5) node {$\lambda$};
\draw (2,-.5)--(5.5,-.5); 
\draw (6,-.4) node {$k$};
\draw (2, -.7)--(2, -.3);
\draw (6.5,-.5)--(9.5,-.5) (1.5,0)--(1.5,2.2);
\draw (9.5, -.7)--(9.5, -.3);
\draw (1.5,2.8) node {$a$};
\draw (1.5,3.3)--(1.5,6);
\draw (1.3,0)--(1.7,0) (1.3,6)--(1.7,6);
\end{tikzpicture}
 \caption{The partition $\boxit{k}{a}{\lambda}$}
\label{figure:complement}
\end{figure}
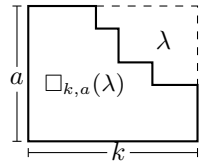

\subsection*{\bf The Littlewood--Richardson coefficients,  $c^{\nu}_{\lambda, \mu}$: } These coefficients are indexed by three partitions.  They are the structure constants in the ring of symmetric functions with respect to the basis of Schur functions. That is,  for any partitions $\lambda$ and $\mu$, 
\begin{equation}\label{eq:LR_coeffs}
s_\lambda s_\mu =\sum_{\nu} c^\nu_{\lambda,\mu} s_\nu.
\end{equation}
These coefficients are important because they occur in many other mathematical contexts. For example,  in representation theory they occur as multiplicities of tensor products of irreducible representations of general linear groups and also as multiplicities in the decompositions of certain induced representations of the symmetric group.  In algebraic geometry,  they occur as structure coefficients when multiplying Schubert classes in the cohomology ring of the Grassmannian.

In Theorem \ref{symmetry:LR} and Proposition \ref{symmetry:LRtranslation} we show that they satisfy 
\begin{equation}
c^{\nu}_{\lambda, \mu} =c^{\boxit{l+m}{n}{\nu}}_{\boxit{l}{n}{\lambda}, \boxit{m}{n}{\mu}},  \quad \text{when $\lambda \subseteq (l^n)$, $\mu \subseteq (m^n)$ and $\nu \subseteq((l+m)^n)$} \label{eqLR}  
\end{equation}
(see Figure \ref{figure:LR})  and
\begin{equation}
c_{\lambda,\mu}^{\nu}  = c_{\lambda+(k^n),\mu}^{\nu+(k^n)},  \quad \text{when $\nu$ and $\lambda$ have length at most $n$.} \label{eqLR:trans}
 \end{equation}

\begin{figure}[ht]
\begin{tikzpicture}[scale=.3]
\tikzstyle{every node}=[inner sep=3pt]; 
\draw[thick] (1,.5)--(1,6)--(4,6)--(4,5)--(5,5)--(5,3.5)--(6.5,3.5)--(6.5,2.5)--(8,2.5)--(8,.5)--cycle;
\draw (3.5,2.5) node { $\boxit{l}{n}{\lambda} $};
\draw[dashed] (4,6)--(8,6)--(8,2.5);
\draw (6.7,4.5) node {$\lambda$};
\draw (1,0)--(4,0); 
\draw (4.5,0) node {$l$};
\draw (8, -.3)--(8, .3);
\draw (1, .3)--(1, -.3);
\draw (5.5,-0)--(8,0) (.3,.5)--(.3,2.4);
\draw (.3,3.1) node {$n$};
\draw (.3,3.7)--(.3,6);
\draw (.1,.5)--(.5,.5) (.1,6)--(.5,6);

\draw[thick] (12,.5)--(12,6)--(15,6)--(15,5)--(16,5)--(16,3.5)--(18.5,3.5)--(18.5,2.5)--(20.5,2.5)--(20.5,.5)--cycle;
\draw (14.5,2.5) node { $\boxit{m}{n}{\mu} $};
\draw[dashed] (15,6)--(20.5,6)--(20.5,2.5);
\draw (18,4.5) node {$\mu$};
\draw (12,-.3)--(15.1,-.3); 
\draw (16, -.3) node {$m$};
\draw (12, 0)--(12, -.6);
\draw (16.9,-.3)--(20.5,-.3) (11.5,.5)--(11.5,2.4);
\draw (20.5, 0)--(20.5, -.6);
\draw (11.4,3.1) node {$n$};
\draw (11.5,3.9)--(11.5,6);
\draw (11.3,0.5)--(11.7,0.5) (11.3,6)--(11.7,6);

\draw[thick] (24,.5)--(24,6)--(31,6)--(31,5)--(32,5)--(32,3.5)--(34.5,3.5)--(34.5,2.5)--(37.5,2.5)--(37.5,.5)--cycle;
\draw (28.5,2.5) node { $\boxit{l+m}{n}{\nu} $};
\draw[dashed] (26,6)--(37.5,6)--(37.5,1);
\draw (34.7,4.4) node {$\nu$};
\draw (23.5,.5)--(23.5,2.4)  (23.5,3.9)--(23.5,6) (23.2,.5)--(23.8,.5) (23.2,6)--(23.8,6);
\draw (23.3,3.1) node {$n$};
\draw (24,-.3)--(28.5,-.3) (32.5,-.3)--(37.5,-.3);
\draw (30.5,-.4) node {$l+m$};
\draw (24, -.6)--(24, .1) (37.5, -.6)--(37.5,-.1);
\end{tikzpicture}
\caption{The partitions occurring in the symmetry of Theorem \ref{symmetry:LR}} 
\label{figure:LR}
\end{figure}
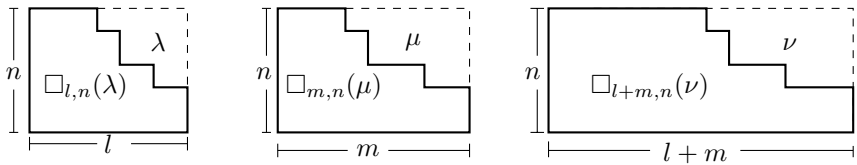

Other symmetries for Littlewood--Richardson, such as $c_{\lambda,\mu}^{\nu}=c_{\mu,\lambda}^{\nu}$ and $c_{\lambda,\mu}^{\nu}=c_{\lambda,\boxit{m}{n}{\nu}}^{\boxit{m}{n}{\mu}}$, have been extensively studied. In particular a number of bijective proofs for them have been found, see \cite{PakVallejoCones,PakVallejoBijections,Thomas-Yong} and the references therein. These other  symmetries, that generate a full symmetric group $S_3$, are  also obvious in the setting of Schubert calculus, since the numbers $c_{\lambda,\nu}^{\boxit{m}{n}{\mu}}$ interpret as triple intersections of Schubert varieties \cite[\S4, Eq. (23)]{Fulton}.  The symmetry \eqref{eqLR} is probably folklore, but we did not find it in the literature in the way presented here.  

It is, however,  equivalent to the symmetry mentioned in \cite[\S2. rem.(a)]{BerensteinZelevinsky}, as the generator of the $\ZZ_2$ subgroup in a $\ZZ_2 \times S_3$ group of symmetries (the factor $S_3$ is the group of other symmetries aforementioned).

Of course, the identity \eqref{eqLR:trans} is very easily established from the combinatorial descriptions of the Littlewood--Richardson coefficients (\emph{e.g.} the Littlewood--Richardson rule).
Similar identities will be shown to hold for Kronecker and plethysm coefficients, for which akin combinatorial descriptions are unavailable, which makes them more difficult to prove without Schur polynomials (or equivalent representation--theoretic considerations).

\subsection*{\bf The Kronecker coefficients, $\kro{\lambda}{\mu}{\nu}$:}  Understanding the Kronecker coefficients is a major open problem in the representation theory of the symmetric and the general linear group. These coefficients also appear naturally  in some interesting problems in quantum information theory \cite{luque-thibon1, luque-thibon2}, geometric complexity theory \cite{BLMW, BOR-complexity} and invariant theory.  We show that they also satisfy similar symmetries  in Theorem \ref{symmetry:kronecker} and Proposition \ref{symmetry:Krontranslation}:
\begin{equation}
\kro{\lambda}{\mu}{\nu}=\kro{\boxit{l}{mn}{\lambda}}{\boxit{m}{l n}{\mu}}{\boxit{n}{l m}{\nu}},   \label{eq:symmetry:kronecker}  
\end{equation}
when $\lambda \subseteq (l^{(mn)})$, $\mu \subseteq (m^{(l n)})$ and $\nu \subseteq(n^{(l m)})$, and
\begin{equation}
\kro{\lambda}{\mu}{\nu} = \kro{\lambda+((km)^l)}{\mu+(kl)^m}{\nu+(k)^{l m}}  \label{eq:symmetry:kronecker:trans}
\end{equation}
for $l$ and $m$ at least the length of $\lambda$, $\mu$ respectively.

Symmetry \eqref{eq:symmetry:kronecker} has also been shown in \cite[Proposition B.1]{StembridgeAppendix}, and identity 
\eqref{eq:symmetry:kronecker:trans}  in \cite[Theorem 3.1]{Vallejo2009}. In both cases, they are established using representation theory.

\subsection*{\bf The Plethysm coefficients, $\pleth{\lambda}{\mu}{\nu}$:}  The plethysm of two symmetric functions $f$ and $g$ is denoted by $f[g]$.  This operation was introduced by Littlewood \cite{Littlewood-plethysm} in the context of compositions of representations of the general linear groups.  Plethysm has been shown to have important applications to physics \cite{Wybourne} and invariant theory \cite{Foulkes}.   
We obtain  two pairs of symmetries for the coefficients $\pleth{\lambda}{\mu}{\nu}$ in the expansion of $s_\lambda[s_\mu]$ into Schur functions (Theorem \ref{symmetry:plethysm} and Proposition \ref{symmetry:plethtrans} on the one hand, Theorem \ref{symmetry:plethysm2} and Proposition \ref{translation:pleth2} on the other hand) :
\begin{align}
\pleth{\lambda}{\mu}{\nu}&=\pleth{\lambda}{\boxit{m}{n}{\mu}}{\boxit{m|\lambda|}{n}{\nu}} \quad \text{when $\mu \subseteq (m^n)$ and $\nu \subseteq ((m|\lambda|)^n)$,} \label{eqPlet} \\
\pleth{\lambda}{\mu}{\nu}&=\pleth{\lambda}{\mu+(k^n)}{\nu+((k|\lambda|)^n)} \quad \text{when $\nu$ has length at most $n$},  \label{eqPlet:trans}     \\
\pleth{\lambda}{\mu}{\nu}&=\pleth{\boxit{l}{r}{\lambda}}{\mu}{\boxit{ql}{n}{\nu}} \quad \text{when $\lambda \subseteq (l^r)$ and $\nu \subseteq ((q l)^n)$,}  \label{eqPlet2} \\
\pleth{\lambda}{\mu}{\nu}&=\pleth{\lambda+(k^r)}{\mu}{\nu+((qk)^n)} \quad \text{when when $\nu$ has length at most $n$}, \label{eqPlet2:trans}
\end{align}
where $r$ is the number of semistandard tableaux of shape $\nu$ filled with numbers 1 through $n$, and $q=r|\nu|/n$.
In both  \eqref{eqPlet} and  \eqref{eqPlet2}, the symmetries that we obtain involves taking complements with respect to two rectangles, see Figure \ref{figure:pleth}.

\begin{figure}[ht] 
\begin{tikzpicture}[scale=.25]
\tikzstyle{every node}=[inner sep=3pt]; 
\draw[thick] (0,0)--(5,0)--(5,1)--(4,1)--(4,2)--(3,2)--(3,4)--(2,4)--(2,5)--(0,5)--cycle;
\draw (1.5,2.5) node {$\lambda$};

\draw[thick] (11,-1)--(11,6)--(14,6)--(14,5)--(15,5)--(15,3.5)--(17.5,3.5)--(17.5,2.5)--(18.5,2.5)--(18.5,-1)--cycle;
\draw (14.3,1.5) node {\small $\boxit{m}{n}{\mu} $};
\draw[dashed] (14,6)--(18.5,6)--(18.5,2.5);
\draw (17,4.5) node {$\mu$};
\draw (11,-1.5)--(13.8,-1.5); 
\draw (14.8,-1.7) node {$m$};
\draw (11, -1.7)--(11, -1.3);
\draw (15.9,-1.5)--(18.5,-1.5) (10.3,-1)--(10.3,1.8);
\draw (18.5, -1.7)--(18.5, -1.3);
\draw (10.3,2.3) node {$n$};
\draw (10.3,3.3)--(10.3,6);
\draw (10,-1)--(10.6,-1) (10,6)--(10.6,6);

\draw[thick] (24,-1)--(24,6)--(32,6)--(32,5)--(34,5)--(34,2.5)--(35.5,2.5)--(35.5,1)--(37.5,1)--(37.5,-1)--cycle;
\draw (30,2) node { $\boxit{m |\lambda|}{n}{\nu} $};
\draw[dashed] (32,6)--(37.5,6)--(37.5,1);
\draw (35.7,4.4) node {$\nu$};
\draw (23,-1)--(23,1.8)  (23,3.2)--(23,6) (22.7,-1)--(23.2,-1) (22.7,6)--(23.2,6);
\draw (23,2.5) node {$n$};
\draw (24,-1.6)--(28.5,-1.6) (32.5,-1.6)--(37.5,-1.6);
\draw (30.5,-1.8) node {$m|\lambda|$};
\draw (24, -1.9)--(24,-1.3) (37.5, -1.9)--(37.5,-1.3);
\end{tikzpicture}

\begin{tikzpicture}[scale=.25]
\tikzstyle{every node}=[inner sep=3pt]; 

\draw[thick] (1,.5)--(1,9)--(4,9)--(4,7.5)--(5,7.5)--(5,6.5)--(6.5,6.5)--(6.5,4.5)--(8.5,4.5)--(8.5,.5)--cycle;
\draw (4,3.5) node { $\boxit{l}{r}{\lambda} $};
\draw[dashed] (4,9)--(8.5,9)--(8.5,2.5);
\draw (7,7.5) node {$\lambda$};
\draw (1,-.1)--(4.5,-.1); 
\draw (5,-.2) node {$l$};
\draw (8.5, -.4)--(8.5, .2);
\draw (1, .2)--(1, -.4);
\draw (5.5,-.1)--(8.5,-.1) (.3,.5)--(.3,3.4);
\draw (.3,4.5) node {$r$};
\draw (.3,5.5)--(.3,9);
\draw (.1,.5)--(.5,.5) (.1,9)--(.5,9);

\draw[thick] (12,0)--(19,0)--(19,1)--(17,1)--(17,3)--(15,3)--(15,4)--(12,4)--cycle;
\draw (14,1.5) node {$\mu$};

\draw[thick] (24,-1)--(24,6)--(34,6)--(34,5)--(35,5)--(35,3)--(35.5,3)--(35.5,2)--(37.5,2)--(37.5,-1)--cycle;
\draw (30,2) node { $\boxit{ql}{a}{\nu} $};
\draw[dashed] (32,6)--(37.5,6)--(37.5,1);
\draw (36.2,4.4) node {$\nu$};
\draw (23,-1)--(23,1.8)  (23,3.2)--(23,6) (22.7,-1)--(23.2,-1) (22.7,6)--(23.2,6);
\draw (23,2.5) node {$n$};
\draw (24,-1.6)--(28.5,-1.6) (32.5,-1.6)--(37.5,-1.6);
\draw (30.5,-1.8) node {$ql$};
\draw (24, -1.9)--(24,-1.3) (37.5, -1.9)--(37.5,-1.3);
\end{tikzpicture}

\caption{The partitions in the symmetry of Theorem \ref{symmetry:plethysm} or Equation \eqref{eqPlet} (top) and  Theorem \ref{symmetry:plethysm2} or Equation \eqref{eqPlet2} (bottom)}
 \label{figure:pleth}
\end{figure}
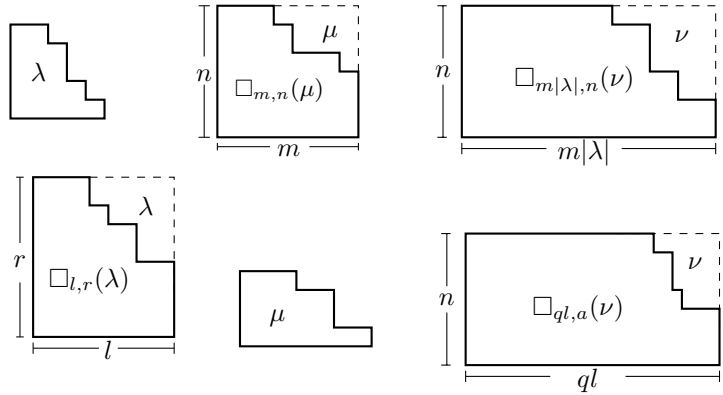

\subsection*{\bf The Kostka-Foulkes polynomials, $K_{\lambda,\mu}(t)$:}   The Kostka numbers, $K_{\lambda, \mu}$, are the coefficients in the decompositions of Schur functions in the basis of monomial functions
\[
s_{\lambda} = \sum_{\mu} K_{\lambda,\mu} \;m_{\mu}.
\]
Let us recall two important interpretations of the Kostka number $K_{\lambda,\mu}$. A combinatorial one:  it counts the number of semistandard tableaux  of shape $\lambda$ and weight $\mu$. And a representation--theoretic one: it is equal to the dimension of the weight space of weight $\mu$ in the irreducible representation $S_{\lambda}(\CC^n)$ of $GL_n(\CC)$. 

Rather than establishing a symmetry theorem for Kostka numbers, we will do it for the more general (one variable) \emph{Kostka-Foulkes polynomials} $K_{\lambda, \mu}(t)$. These are deformations of the Kostka numbers, which are recovered by evaluating the Kostka--Foulkes polynomials at $t=1$. The Kostka--Foulkes polynomials are the coefficients in the decompositions of Schur functions in the basis of Hall--Littlewood polynomials $P_\mu(X;t)$:
\begin{equation}\label{eq:stoP}
s_{\lambda}[X] = \sum_{\mu} K_{\lambda,\mu}(t) \;P_\mu(X;t).
\end{equation}
In Theorem \ref{symmetry:Kostka1var} and Proposition \ref{symmetry:Kostka1var:trans} we have obtained the following rectangular symmetries for the Kostka-Foulkes polynomials: 
\begin{align}
K_{\lambda, \mu}(t) &= K_{\boxit{k}{n}{\lambda}, \boxit{k}{n}{\mu}}(t) \quad \text{ when $\lambda, \mu \subseteq (k^n)$, and } \label{eq:Kostka}\\
K_{\lambda, \mu}(t) &= K_{\lambda+(k^n), \mu+(k^n)}(t) \quad \text{when $\mu$ has length at most $n$.}  \label{eq:Kostka:trans}
\end{align}
These identities still hold for the Kostka numbers $K_{\lambda,\mu}=K_{\lambda,\mu}(1)$, by specialization at $t=1$. 

The bijection in the solution of exercise 7.41 in \cite{StanleyBookEC2} gives a bijective proof of \eqref{eq:Kostka} for the Kostka numbers.  But  this does not generalize to a bijective proof for the identity for the Kostka-Foulkes polynomials. 

Note that it is immediate to obtain a bijective proof of \eqref{eq:Kostka:trans} for the Kostka numbers, that is easy to adapt for the Kostka--Foulkes polynomials (charactarized combinatorially as the generating function for the \emph{charge} of semistandard Young tableaux). 

The symmetry \eqref{eq:Kostka} may also be deduced from a much more elaborate result by Shimozono and Weyman on the Poincar{\'e} polynomials of graded characters of isotopic components of a natural family of $GL(\CC^n)$--modules supported in the closure of a nilpotent conjugacy class \cite[Eq. (2.16)]{ShimozonoWeyman}.


\section{Algebraic tools}
We assume that the reader is familiar with the various algebraic structures on  the space of symmetric functions, $\Sym$,
and in particular, with its main bases. For background information see \cite{LascouxBook, MacdonaldBook, StanleyBookEC2}.  We mainly follow the notation of \cite{StanleyBookEC2}, except for the fact that we draw our Ferrers diagrams using the French notation. 

Let $\PP{n}$ the set of all weakly decreasing sequences $(\lambda_1,\lambda_2,\ldots,\lambda_n)$ of nonnegative integers. When dealing with weakly decreasing sequences of integers, it will be convenient not to distinguish between sequences that differ only by trailing zeros. Therefore  $\PP{n}$ represents as well the set of integer partitions with length at most $n$. Given any two integer partitions $\lambda$ and $\mu$, $\lambda \subseteq \mu$ stands for the inclusion of the corresponding Ferrers diagrams, $\lambda'$ is the conjugate of $\lambda$, and $\lambda+\mu$ is the partition whose parts are the $\lambda_i+\mu_i$.  We use $\ell(\lambda)$ to denote the number of nonzero parts of  $\lambda$, i.e., its \emph{length}. Last, $(k^n)$ stands for the sequence with $n$ terms all equal to $k$.

Let $X=\{x_1,x_2,\ldots\}$ be a countable set of independent variables.  For $n \geq 0$, we set $X_n=\{x_1,x_2,\ldots,x_n\}$. The ring of symmetric polynomials, $\ZZ[x_1,x_2,\ldots,x_n]^{S_n}$, admits as a linear basis the Schur polynomials, $s_{\lambda}[X_n]=s_{\lambda}(x_1,x_2,\ldots,x_n)$, indexed by all $\lambda \in \PP{n}$. They are defined by
\begin{equation}\label{JacobiBialternant}
s_{(\lambda_1,\lambda_2,\ldots,\lambda_n)}[X]=\frac{\det(x_i^{\lambda_j+j-1})_{1\leq i,j\leq n}}{\det(x_i^{j-1})_{1\leq i,j\leq n}}.
\end{equation}
This is Jacobi's definition of Schur polynomials as ``bialternants" \cite[I.\S3.(3.1)]{MacdonaldBook}.

Let us consider now $\ZZ[x_1^{\pm 1},x_2^{\pm 1}, \ldots, x_n^{\pm 1}]^{S_n}$, the ring of symmetric Laurent polynomials in $n$ variables. Let $\P{n}$ be the set of all weakly decreasing sequences of integers $(\lambda_1,\lambda_2,\ldots,\lambda_n)$. (Compared to the definition of $\PP{n}$, we dropped the requirement of nonnegativity).  We define the \emph{Schur Laurent polynomials}  $s_{\lambda}[X_n]$, for $\lambda \in \P{n}$, again by \eqref{JacobiBialternant}. 
Denote by $\inv{X_n}$ the set of the inverses of the variables, i.e.,  $\inv{X_n}=\{\frac{1}{x_1},\frac{1}{x_2}, \ldots, \frac{1}{x_n}\}$. 
For any sequence $\lambda \in \P{n}$  and any integer $k$, define the new sequence
\[
\boxit{k}{n}{\lambda}=(k-\lambda_n, k-\lambda_{n-1}, \ldots, k-\lambda_1).
\]
This sequence is also in $\P{n}$. This extends the definition given in the introduction, when $\lambda$ is a partition that fits in the diagram of $(k^n)$. In that case, $\boxit{k}{n}{\lambda}$ is also a partition, ``complement''  of $\lambda$ in the rectangle.

It is immediate to check from \eqref{JacobiBialternant} the following properties:
\begin{lemma}
For all $\lambda \in \P{n}$ and all integers $k$, we have
\begin{equation}\label{translation}
s_{\lambda+(k^n)}[X_n]=(x_1 x_2 \cdots x_n)^k s_{\lambda}[X_n]. 
\end{equation}
and 
\begin{equation}\label{schur_of_inverses}
s_{\lambda}[\inv{X_n}]=s_{\boxit{0}{n}{\lambda}}[X_n].
\end{equation}
\end{lemma}
Formula \eqref{schur_of_inverses} is well known, see  \cite[(I.4.12.)]{LascouxBook}, \cite[Ex. 7.41]{StanleyBookEC2}  or \cite[B]{StembridgeAppendix}.

\begin{proposition}
The Schur Laurent polynomials $s_{\lambda}[X_n]$, for $\lambda \in \P{n}$, are a basis for $\ZZ[x_1^{\pm 1},x_2^{\pm 1}, \ldots, x_n^{\pm 1}]^{S_n}$. 
\end{proposition} 
\begin{proof}
Any  $g\in \ZZ[x_1^{\pm 1},x_2^{\pm 1}, \ldots, x_n^{\pm 1}]^{S_n}$ can be written in the form $f/(x_1 x_2 \cdots x_n)^k$ for $f$ a symmetric polynomial and $k$ an integer. The polynomial $f$ expands as a linear combination of Schur polynomials. Dividing by $(x_1 x_2 \cdots x_n)^k$ expresses $g$ as a linear combination of Schur Laurent polynomials by \eqref{translation}. Therefore the Schur Laurent polynomials generate $\ZZ[x_1^{\pm 1},x_2^{\pm 1}, \ldots, x_n^{\pm 1}]^{S_n}$.

Given any linear relation between Schur Laurent polynomials, we obtain a linear relation between Schur polynomials, with the same coefficients, by multiplying by a big enough power of $(x_1 x_2 \cdots x_n)$. This shows that the Schur Laurent polynomials are linearly independent.
\end{proof}

\begin{remark}[Representation--theoretic interpretation of \eqref{schur_of_inverses}]
Let $V$ be a complex vector space of dimension $n$.
The Schur polynomials in $n$ variables are the formal characters of the irreducible polynomial representations of $GL(V)$, the $\mathbb{S}_{\lambda}(V)$ for $\lambda \in \PP{n}$. The Schur Laurent polynomials are the formal characters of its \emph{rational} irreducible representations.

Relation \eqref{translation} corresponds to the isomorphism 
\[
\mathbb{S}_{\lambda+(k^n)}(V) \cong \mathbb{S}_{\lambda} \otimes D_k
\]
where $D_k$ is the one dimensional representation where $g\in GL(V)$ acts as the multiplication by $\det(g)^k$.

The Schur Laurent polynomial $s_{\lambda}[\inv{X_n}]$ is the formal character of the dual representation $\mathbb{S}_{\lambda}(V^*)$. The identity \eqref{schur_of_inverses} means that
\[
\mathbb{S}_{\lambda}(V^*) \cong \mathbb{S}_{\boxit{0}{n}{\lambda}}(V).
\]

\end{remark}

We will now exploit \eqref{translation} and  \eqref{schur_of_inverses} systematically to produce symmetries for the Littlewood--Richardson coefficients, the Kronecker coefficients and the plethysm coefficients. In Section \ref{sec:Kostka}, we will extend   \eqref{translation} and  \eqref{schur_of_inverses}  to Hall--Littlewood polynomials, to produce symmetries for the Kostka--Foulkes polynomials.


\section{Littlewood--Richardson coefficients}  
In this section we will prove the rectangular symmetries for the Littlewood-Richardson coefficients.  
Let $n$ be  a nonnegative integer and $\lambda$ and $\mu$ be two partitions. 
If we specialize \eqref{eq:LR_coeffs} at $X_n=\{x_1,x_2,\ldots,x_n\}$, we get
\begin{equation}\label{eqLRn}
s_{\lambda}[X_n] s_{\mu}[X_n] = \sum_{\nu: \ell(\nu) \leq n} c^{\nu}_{\lambda,\mu} s_{\nu}[X_n].
\end{equation}
If $\ell(\lambda)$ or $\ell(\mu)$ is bigger than $n$, then the right--hand side is zero. Then all coefficients $c^{\nu}_{\lambda,\mu}$ in the right--hand side are zero. We assume now that $\lambda$ and $\mu$ have length at most $n$.
Let us replace each $x_i$ with $1/x_i$. We obtain
\[
s_{\lambda}[\inv{X_n}] s_{\mu}[\inv{X_n}] = \sum_{\nu: \ell(\nu) \leq n} c^{\nu}_{\lambda,\mu} s_{\nu}[\inv{X_n}].
\]
By \eqref{schur_of_inverses}, this can be written as
\[
s_{\boxit{0}{n}{\lambda}}[X_n] s_{\boxit{0}{n}{\mu}}[X_n] = \sum_{\nu: \ell(\nu) \leq n} c^{\nu}_{\lambda,\mu} s_{\boxit{0}{n}{\nu}}[X_n].
\]
Let $l \geq \lambda_1$ and $m \geq \mu_1$.  Let us multiply both sides with $(x_1 x_2 \cdots x_n)^{l+m}$. We get, by \eqref{translation},
\[
s_{\boxit{l}{n}{\lambda}}[X_n] s_{\boxit{m}{n}{\mu}}[X_n] = \sum_{\nu: \ell(\nu) \leq n} c^{\nu}_{\lambda,\mu} s_{\boxit{l+m}{n}{\nu}}[X_n].
\]
Now, let $\nu$ be a partition with length at most $n$. If $\nu_1 > l + m$, then $s_{\boxit{l+m}{n}{\nu}}[X_n]$ is not a Schur polynomial, and thus does not appear in this expansion. In this case $c_{\lambda,\mu}^{\nu}=0$. Else 
\[
c^{\nu}_{\lambda,\mu}=c_{\boxit{l}{n}{\lambda},\boxit{m}{n}{\mu}}^{\boxit{l+m}{n}{\nu}}.
\]

We have proved the following theorem.
\begin{theorem}\label{symmetry:LR}
Let $l$, $m$, $n$ be nonnegative integers and  $\lambda$, $\mu$ and $\nu$ be three partitions such that $\ell(\nu) \leq n$, $\lambda_1 \leq l$ and $\mu_1 \leq m$.  If $\lambda \subseteq (l^n)$, $\mu \subseteq (m^n)$ and $\nu \subseteq ((l+m)^n)$ then
\begin{equation}
c^{\nu}_{\lambda, \mu}=c^{\boxit{l+m}{n}{\nu}}_{\boxit{l}{n}{\lambda}, \boxit{m}{n}{\mu}}.\tag{\ref{eqLR}}
\end{equation}
Else $c^{\nu}_{\lambda, \mu}=0$.
\end{theorem}

Note that the involution $\omega$ (that sends the elementary symmetric function $e_i$ to the complete sum $h_i$, see \cite{MacdonaldBook} I.\S2) produces  a similar symmetry with respect to three rectangles of the same width, instead of height.

In addition, we obtain the following translational symmetry for $c_{\lambda, \mu}^\nu$ by multiplying \eqref{eqLRn} and  $(x_1 x_2 \cdots x_n)^k$ and using \eqref{translation}. 

\begin{proposition}\label{symmetry:LRtranslation}
  Let $n \geq 0$ and $k$ be integers and $\lambda$, $\mu$, $\nu$ be partitions such that $n \geq \ell(\nu)$ and $\lambda+(k^n)$ is a partition (\emph{i.e} $\lambda_n+k \geq 0$).  If $\ell(\lambda) \leq n$ and $\nu+(k^n)$ is a partition, then
\begin{equation}
c_{\lambda,\mu}^{\nu}=c_{\lambda+(k^n),\mu}^{\nu+(k^n)}, \tag{\ref{eqLR:trans}}
\end{equation}
else $c_{\lambda,\mu}^{\nu}=0$.
\end{proposition}

By means of the symmetry $c_{\lambda,\mu}^{\nu}=c_{\mu,\lambda}^\nu$, we obtain  $c_{\lambda,\mu}^{\nu}=c_{\lambda,\mu+(k^l)}^{\nu+(k^l)}$.


\section{Kronecker coefficients}

 In this section, we apply Formula (\ref{schur_of_inverses}) to derive a rectangular symmetry for the Kronecker coefficients. The  symmetry \eqref{symmetry:kronecker} was also found by Stembridge \cite{StembridgeAppendix}. The method is basically the same, except for the presentation: where  Stembridge uses representations of general linear groups we use their formal characters (symmetric Laurent polynomials).


We start with the following description of the Kronecker cofficients: let $X$ and $Y$ be two independant set of variables $x_1,x_2,\ldots$ and $y_1,y_2,\ldots$. Let $f[XY]$ stand for the evaluation of the symmetric function $f$ at all products $x_i y_j$, this is a symmetric function in $X$ and in $Y$ and expands in the basis of the $s_{\lambda}[X] s_{\mu}[Y]$. Then, for all partitions $\nu$, (see \cite[I.\S7.(7.9)]{MacdonaldBook})
\[
s_{\nu}[XY]=\sum_{\lambda, \mu} \kro{\lambda}{\mu}{\nu} s_{\lambda}[X] s_{\mu}[Y].
\] 

Let $l$ and $m$ be  nonnegative integers and $\nu$ a partition.
We specialize the above identity at the finite sets of variables $X_{l}$ and $Y_m$ and we get
\begin{equation}\label{kron_lm}
s_{\nu}[X_l Y_m]=\sum \kro{\lambda}{\mu}{\nu} s_{\lambda}[X_l] s_{\mu}[Y_m],
\end{equation}
where the sum is over all pairs of partitions $\lambda$ and $\mu$ such that $\ell(\lambda) \leq l$ and $\ell(\mu) \leq m$. If $l m > \ell(\nu)$ then the right--hand side is zero. In this case all coefficients in the expansion are zero.  We now assume that $l m \leq \ell(\nu)$. We replace each variable by its inverse. Since $\inv{(X_{l} Y_m)}=\inv{X_{l}} \inv{Y_m}$, we get
\[
s_{\nu}[\inv{X_l} \inv{Y_m}]=\sum \kro{\lambda}{\mu}{\nu} s_{\lambda}[\inv{X_l}] s_{\mu}[\inv{Y_m}].
\]
By \eqref{schur_of_inverses}, this means that 
\[
s_{\boxit{0}{l m}{\nu}}[X_l Y_m] = \sum \kro{\lambda}{\mu}{\nu} s_{\boxit{0}{l}{\lambda}}[X_l] s_{\boxit{0}{m}{\mu}}[Y_m].
\]
Let $n \in \ZZ$. We multiply the previous identity by $(\prod_{i,j} x_i y_j)^n$. Note that  
\[
\prod_{i,j} x_i y_j = (x_1 x_2 \cdots x_l)^m (y_1 y_2 \cdots y_m)^{l},
\]
thus
\[
\left(\prod_{i,j} x_i y_j \right)^n =  (x_1 x_2 \cdots x_l)^{mn} (y_1 y_2 \cdots y_m)^{l n}.
\]
By \eqref{translation}, we get
\[
s_{\boxit{n}{l m}{\nu}}[X_l Y_m] = \sum \kro{\lambda}{\mu}{\nu} s_{\boxit{mn}{l}{\lambda}}[X_l] s_{\boxit{l n}{m}{\mu}}[Y_m].
\]
Let us assume now that $n \geq \nu_1$, so that the right--and side is a Schur polynomial. We see that if $\lambda \not \subset ((mn)^{l}) $ or $\mu \not \subset ((l n)^m)$ then $\kro{\lambda}{\mu}{\nu}=0$, and else
\[
\kro{\lambda}{\mu}{\nu} = \kro{\boxit{mn}{l}{\lambda}}{\boxit{l n}{m}{\mu}}{\boxit{n}{l m}{\nu}}
\]
We reformulate this in a more symmetric way by applying it rather to the conjugates of $\lambda$ and $\mu$:
\begin{align*}
\kro{\lambda'}{\mu'}{\nu} 
&= \kro{\boxit{mn}{l}{\lambda'}}{\boxit{l n}{m}{\mu'}}{\boxit{n}{l m}{\nu}}\\
&= \kro{(\boxit{l}{mn}{\lambda})'}{(\boxit{m}{l n}{\mu})'}{\boxit{n}{l m}{\nu}}\\
\end{align*}
We also use the identity $\kro{\alpha'}{\beta'}{\gamma}=\kro{\alpha}{\beta}{\gamma}$ for any three partitions. We get the following theorem.

\begin{theorem}\label{symmetry:kronecker}
Let $l$, $m$ and $n$ be three nonnegative integers and $\lambda$, $\mu$ and $\nu$ be three partitions such that $\lambda_1 \leq l$, $\mu_1 \leq m$, $\nu_1 \leq n$.   If $\lambda \subseteq (l^{mn})$, $\mu \subseteq (m^{l n})$ and $\nu \subseteq (n^{l m})$, then 
\begin{equation}
\kro{\lambda}{\mu}{\nu}
=\kro{\boxit{l}{mn}{\lambda}}{\boxit{m}{l n}{\mu}}{\boxit{n}{l m}{\nu}}. \tag{\ref{eq:symmetry:kronecker}}
\end{equation}
Else $\kro{\lambda}{\mu}{\nu}=0$.
\end{theorem}

Similarly as for the Littlewood-Richardson coefficients, we also obtain the following translational symmetry for the Kronecker coefficients, by multiplying \eqref{kron_lm} by $(\prod_{i,j} x_i y_j)^n$ and using \eqref{translation}.

\begin{proposition}\label{symmetry:Krontranslation}
Let $\lambda$, $\mu$ and $\nu$ be partitions. Let $l \geq 0$, $m \geq 0$ and $k \in \ZZ$ be integers such that 
$l \geq \ell(\lambda)$, $m \geq \ell(\mu)$ and $\nu+(k^{mn})$ is a partition (\emph{i.e.} has no negative components). If $\ell(\nu) \leq l m$ and $\lambda+((km)^{l})$ and $\mu+((kl)^m)$ are partitions, then 
\begin{equation}
\kro{\lambda}{\mu}{\nu} = \kro{\lambda+((km)^l)}{\mu+(kl)^m}{\nu+(k)^{l m}} \tag{\ref{eq:symmetry:kronecker:trans}}
\end{equation}
and else $\kro{\lambda}{\mu}{\nu} =0$.
\end{proposition}

\subsection*{The case of three rectangles}  An important class of Kronecker coefficients are those indexed by rectangular partitions. They are important in quantum information theory to model entanglement \cite{luque-thibon1, luque-thibon2} and also to advance the program of Geometric Complexity Theory \cite{BLMW}.  

The following corollary also appears in\cite[(C.1)]{StembridgeAppendix}.
\begin{corollary} \label{rectangles}
Let $k$ and $d$ be nonnegative integers.  If $k \leq d^2$, 
\begin{equation}\label{exterior}
g( (d^k), (d^k), (d^k)) = g( (d^{d^2-k}), (d^{d^2-k}), (d^{d^2-k}))
\end{equation}
and when $k>d^2$,  this Kronecker coefficient is zero. 
\end{corollary}
\begin{proof}
Set $l=m=n=d$ in Theorem \ref{symmetry:kronecker} and $\lambda=\mu=\nu=(d^k)$.
\end{proof}

\begin{remark}[Representation--theoretic interpretation of Corollary \ref{rectangles}.] 
Let $V$ be a complex vector space of dimension $d$. Consider the exterior algebra:
\[
\Lambda\left(V \otimes V \otimes V\right)=\bigoplus_{i=0}^{d^3} \Lambda^i\left(V \otimes V \otimes V\right)
\]
The group $GL(V) \times GL(V) \times GL(V)$ acts on this exterior algebra. The Kronecker coefficient $\kro{\lambda'}{\mu'}{\nu'}$ is the multiplicity of its irreducible representation $\mathbb{S}_{\lambda}(V) \otimes\mathbb{S}_{\mu}(V) \otimes\mathbb{S}_{\nu}(V) $. In particular, $\Lambda^i\left(V \otimes V \otimes V\right)$ contains non--trivial invariants for $SL(V) \times SL(V) \times SL(V)$ if and only if there exists an integer $k$ such that $i=kd$. Then the dimension of the subspace of invariants  is the rectangular Kronecker coefficient $\kro{(d^k)}{(d^k)}{(d^k)}$. Equation \eqref{exterior} follows from the $SL(V \otimes V \otimes V)$ natural isomorphism
\[
\Lambda^{i}\left(V^* \otimes V^* \otimes V^*\right) \cong \Lambda^{d^3-i}\left(V \otimes V \otimes V\right).
\]
\end{remark}

\begin{example}[Weight reduction for Kronecker coefficients]\label{computational:kro}
The naive algorithm to compute a Kronecker coefficient $\kro{\lambda}{\mu}{\nu}$ consists in converting Schur functions in power sums. Indeed, in the power sums basis, Kronecker products are trivial. 
This is the algorithm used, for instance, currently in SAGE \cite{SAGE} and the Maple package SF \cite{SF} . The cost of the computation depends then mainly on the weight  of $\lambda$, $\mu$ and $\nu$. (Note that other algorithms are available and efficient for partitions of short height, regardless of the weight, see for instance \cite{Christandl}). 

Theorem \ref{symmetry:kronecker} shows that $\kro{\lambda}{\mu}{\nu}$ is equal to other Kronecker coefficients, that may be of smaller weight. Precisely, let $N=|\lambda|=|\mu|=|\nu|$. Then the weight of $\kro{\boxit{l}{mn}{\lambda}}{\boxit{m}{l n}{\mu}}{\boxit{n}{l m}{\nu}}$ (i.e. the weight of the indexing partitions) is $l m n -N$. We can take $l=\lambda_1$, $m=\mu_1$ and $n=\nu_1$, the computation is reduced to the computation of a Kronecker coefficient of weight $\lambda_1 \mu_1 \nu_1 - N$. Last, we may make use of the symmetries under conjugation 
\[
\kro{\lambda}{\mu}{\nu}=\kro{\lambda}{\mu'}{\nu'}=\kro{\lambda'}{\mu}{\nu'}=\kro{\lambda'}{\mu'}{\nu}
\]
to reduce the computation to the computation of a Kronecker coefficient whose weight is the smallest among
\[
K \frac{\lambda_1}{\ell(\lambda)}-N,\quad
 K \frac{\mu_1}{\ell(\mu)}-N, \quad
K \frac{\nu_1}{\ell(\nu)}-N, \quad
\text{ and } K \frac{\lambda_1 \mu_1 \nu_1}{\ell(\lambda) \ell(\mu) \ell(\nu)}-N
\]
where $K$ stands for $\ell(\lambda) \ell(\mu) \ell(\nu)$.
\end{example}


\section{Plethysm coefficients}

The plethysm coefficients are the coefficients  $ \pleth{\lambda}{\mu}{\nu}$ of the plethysms of two Schur functions, expanded in the Schur basis:
\begin{equation}\label{plethysm}
s_{\lambda}[s_{\mu}]=\sum_{\nu} \pleth{\lambda}{\mu}{\nu} s_{\nu}.
\end{equation}
While there are algorithms for  computing $ \pleth{\lambda}{\mu}{\nu}$ (see for example \cite{Chen-Garsia-Remmel, Yang}), no satisfying combinatorial description has been found.  In this section we describe two rectangular symmetries satisfied by the plethysm coefficients. 

\subsection{Preliminaries on plethysm}

It will be useful to extend the plethysm operation to the case when $f$ is a symmetric function but $g=g(x_1,x_2,\ldots)$ is any formal series. This is done by means of the following  two rules: 
\begin{enumerate}
\item the map $f \mapsto f[g]$ is a morphism of algebras.
\item for any positive integer $n$, $p_n[g]=g(x_1^n,x_2^n,\ldots)$ (here $p_n$ is the $n$--th power sum symmetric function).
\end{enumerate}
This determines $f[g]$ for any symmetric function $f$, since the algebra of symmetric functions with rational coefficients is freely generated by the power sums $p_n$. When $g$ is a symmetric function, $f[g]$ defined as above coincides with the plethysm of $f$ with $g$, see \cite[I.\S8.]{MacdonaldBook} or \cite[Def. A.2.6]{StanleyBookEC2}.

This allows us to write, when specializing a plethysm of symmetric functions $f[g]$ to any set of variables $Y$ (in particular finite): 
\[
(f[g])[Y]=f[g[Y]]
\]
where the left--hand side is a plethysm of symmetric function, specialized at a set of variables $Y$, and the right--hand side is an ``extended plethysm'' of the symmetric function $f$ with the formal series $g(Y)$.

We will make use of the following property.
\begin{lemma}\label{lemma:homogeneous}
Let $f$ be a homogeneous symmetric function of degree $L$ and $g(x_1,x_2,\ldots)$ be a formal series. Let $x^w$ be a monomial in $x$. Then
\[
f[x^w g]=x^{L w} f[g].
\]
\end{lemma}
\begin{proof}
It is straightforward to check this when $f$ is a power sum $p_n$, and then to extend this to any symmetric function $f$ using that $f \mapsto f[g]$ is a morphism of algebras.
\end{proof}

\subsection{First pair of symmetries for plethysm coefficients}

Let $X_n=\{x_1, x_2, \ldots, x_n\}$ be a set of $n$ variables, where 
$n$ is a nonnegative integer, and let $\lambda$ and $\mu$ be two partitions. Evaluating \eqref{plethysm} at $X_n$ we get
\[
\left(s_{\lambda}[s_{\mu}]\right)[X_n]=\sum_{\nu\,:\, \ell(\nu) \leq n} \pleth{\lambda}{\mu}{\nu} s_{\nu}[X_n].
\]
But $\left(s_{\lambda}[s_{\mu}]\right)[X_n]=s_{\lambda}[s_{\mu}[X_n]]$. Thus
\[
s_{\lambda}[s_{\mu}[X_n]]=\sum_{\nu\,:\, \ell(\nu) \leq n} \pleth{\lambda}{\mu}{\nu} s_{\nu}[X_n].
\]
If $\ell(\mu) > n$, then the left--hand side is zero, and thus all coefficients $\pleth{\lambda}{\mu}{\nu}$ with $\ell(\nu) \leq n$ are zero. We assume now that $\ell(\mu) \leq n$. 
Replacing each variable by its inverse, we get
\begin{equation}\label{plethysm:inverses}
s_{\lambda}[s_{\mu}[\inv{X_n}]]=\sum_{\nu\,:\, \ell(\nu) \leq n} \pleth{\lambda}{\mu}{\nu} s_{\nu}[\inv{X_n}].
\end{equation}
By \eqref{schur_of_inverses}, we obtain
\[
s_{\lambda}[s_{\boxit{0}{n}{\mu}}[X_n]]=\sum_{\nu\,:\, \ell(\nu) \leq n} \pleth{\lambda}{\mu}{\nu} s_{\boxit{0}{n}{\nu}}[X_n].
\]
Let $m \geq \mu_1$. We multiply both sides by $(x_1 x_2 \cdots x_n)^{L m}$, where $L=|\lambda|$. 
Lemma \ref{lemma:homogeneous}  implies that on the left hand side we have
\begin{align*}
(x_1 x_2 \cdots x_n)^{L m} s_{\lambda}[s_{\boxit{0}{n}{\mu}}[X_n]]
&=
s_{\lambda}[ (x_1 x_2 \cdots x_n)^{m} s_{\boxit{0}{n}{\mu}}[X_n]]\\
&=
s_{\lambda}[ s_{\boxit{m}{n}{\mu}}[X_n]].
\end{align*}
Therefore, we obtain that 
\[
s_{\lambda}[s_{\boxit{m}{n}{\mu}}[X_n]]
= 
\sum_{\nu\,:\, \ell(\nu) \leq n} \pleth{\lambda}{\mu}{\nu} s_{\boxit{m L}{n}{\nu}}[X_n].
\]
Let $\nu$ be a partition such that $\ell(\nu) \leq n$. We see that if $\nu \not \subset ((mL)^n)$ then $\pleth{\lambda}{\mu}{\nu}=0$. Otherwise,
\[
\pleth{\lambda}{\mu}{\nu} = \pleth{\lambda}{\boxit{m}{n}{\mu}}{\boxit{m L}{n}{\nu}}.
\]
We have proved the following result.

\begin{theorem}\label{symmetry:plethysm} Fix nonnegative integers $m$ and $n$ and 
let $\lambda$,  $\mu$ and $\nu$ be partitions such that   $\mu \subseteq (m^n)$ and  $\ell(\nu) \le n$. 
If $\nu \subseteq ((m|\lambda|)^n)$,  then 
\begin{equation}
\pleth{\lambda}{\mu}{\nu} = \pleth{\lambda}{\boxit{m}{n}{\mu}}{\boxit{m |\lambda|}{n}{\nu}}. \tag{\ref{eqPlet}}
\end{equation}
Otherwise $\pleth{\lambda}{\mu}{\nu}=0$.
\end{theorem}

\begin{proposition}\label{symmetry:plethtrans}
Let $\lambda$, $\mu$ and $\nu$ be partitions. Let $n \geq 0$ and $k \in \ZZ$ be integers such that $\ell(\nu) \leq n$ and $\mu+(k^n)$ is a partition.

If  $\nu+((k|\lambda|)^n)$ is a partition, then
\begin{equation}
\pleth{\lambda}{\mu}{\nu}=\pleth{\lambda}{\mu+(k^n)}{\nu+((k|\lambda|)^n)}, \tag{\ref{eqPlet:trans}}
\end{equation}
and else $\pleth{\lambda}{\mu}{\nu}=0$.
\end{proposition}


\subsection{Second pair of symmetries for plethysm coefficients}

There is another way to exploit the alphabet of inverses for plethysm coefficients in order to obtain another rectangular symmetry. Recall that the combinatorial definition of Schur functions says that,  
\[
s_{\mu}=\sum_{T} x^{w(T)}
\]
where the sum is carried over all semistandard tableaux $T$ of shape $\mu$. The exponent $w(T)$ is the weight of $T$, i.e., its $i$--th component is the number of occurrences of $i$ in $T$.  For details see \cite[\S 7.10.1.]{StanleyBookEC2}.

As a consequence, 
\[
s_{\mu}[X_n]=\sum_{T \in \tab{\mu}{n}} x^{w(T)}
\]
where $\tab{\mu}{n}$ is the set of all semistandard tableaux of shape $\mu$ with entries in $\{1,2,\ldots,n\}$.

If $f$ is a symmetric function and $g$ a sum of monomials, then $f[g]$ is the specialization of $f$ at these monomials \cite[I.\S8]{MacdonaldBook}:
\[
f[x^{\omega_1}+x^{\omega_2}+\ldots]=f(x^{\omega_1},x^{\omega_2},\ldots).
\]
This holds in particular for $g=s_{\mu}[X]$.

Let $n$ be a nonnegative integer and $\lambda$ and $\mu$ be partitions.
We consider again \eqref{plethysm:inverses} but we expand $s_{\mu}[\inv{X_n}]$ in monomials. We get that
\begin{align*}
s_{\lambda}[s_{\mu}[\inv{X_n}]]&=s_{\lambda}\left[\sum_{T \in \tab{\mu}{n}} \frac{1}{x^{w(T)}}\right].
\end{align*}
Let $r=\# \tab{\mu}{n}$. Let us assume that $r\geq \ell(\lambda)$. We introduce a new set of variables $Y_r=\{y_1, y_2, \ldots, y_r\}$. Then $\left(s_{\lambda}[s_{\mu}[\inv{X_n}]]\right)$ is equal to the specialization of $s_{\lambda}[\inv{Y}]$ at the monomials $x^{w(T)}$ for $T \in \tab{\mu}{n}$. 
But 
\[
s_{\lambda}[\inv{Y_r}]=s_{\boxit{0}{r}{\lambda}}[Y_r].
\]
Let $l \geq \lambda_1$. Multiplying both sides with $(y_1 y_2 \cdots y_r)^{l}$ we get 
\[
(y_1 y_2 \cdots y_r)^{l} s_{\lambda}[\inv{Y_r}]=s_{\boxit{l}{r}{\lambda}}[Y_r].
\]
Let us now specialize at the monomials $x^{w(T)}$ for $T \in \tab{\nu}{n}$. We get
\[
\left( \prod_{T \in \tab{\mu}{n}} x^{w(T)}\right)^{l} s_{\lambda}[s_{\mu}[\inv{X_n}]] = s_{\boxit{l}{r}{\lambda}}[s_{\mu}[X_n]].
\]
Let us examine the product of the monomials $x^{w(T)}$. This product is symmetric in the variables $x_i$ (since the sum of the same monomials is symmetric, being equal to $s_{\mu}[X]$). Therefore there exists an integer $q$ such that
\[
\prod_{T \in \tab{\mu}{n}} x^{w(T)}=(x_1 x_2 \cdots x_n)^q
\]
Let us compare the total degrees on both sides. On the right--hand side, it is $qn$. For the left--hand side, let $M=|\mu|$. Observe that the total degree of each monomial $x^{w(T)}$ is $M$. There are $r$ monomials in the product. Thus, the total degree of the left--hand side is $r M$. This yields the equation $r M = q n$. From this we extract $q=r M/n$. We have now
\[
(x_1 x_2 \cdots x_n)^{q l} s_{\lambda}[s_{\mu}[\inv{X_n}]] = s_{\boxit{l}{r}{\lambda}}[s_{\mu}[X_n]].
\]
Thus
\begin{align*}
s_{\boxit{l}{r}{\lambda}}[s_{\mu}[X_n]]&=\sum_{\nu\,:\, \ell(\nu) \leq n} \pleth{\lambda}{\mu}{\nu} s_{\nu}[\inv{X_n}] \cdot (x_1 x_2 \cdots x_n)^{q l} ,\\
&= 
\sum_{\nu\,:\, \ell(\nu) \leq n} \pleth{\lambda}{\mu}{\nu} s_{\boxit{q l}{n}{\nu}}[X_n].
\end{align*}
From this we deduce that for $\nu$ partition with length at most $n$, if  $\nu \not \subset ((ql)^n)$ then $\pleth{\lambda}{\mu}{\nu}=0$ and else
\[
\pleth{\lambda}{\mu}{\nu} = \pleth{\boxit{l}{r}{\lambda}}{\mu}{\boxit{q l}{n}{\nu}}
\]

We have proved the following theorem.

\begin{theorem}\label{symmetry:plethysm2}
Let $l$ and $n$ be nonnegative integers and $\mu$, $\nu$, and $\lambda$  be partitions such that 
 $ \lambda_1 \leq l,$ and $ \ell(\nu) \leq n$.  Let $r$ be the number of semistandard tableaux of shape $\mu$ and entries in $\{1,2,\ldots,n\}$.  Then $q=r |\mu|/n$ is an integer, and we have that  if $\lambda \subseteq (l^r)$ and $\nu \subseteq ((q l)^n)$,
\begin{equation}
\pleth{\lambda}{\mu}{\nu} = \pleth{\boxit{l}{r}{\lambda}}{\mu}{\boxit{q l}{n}{\nu}}, \tag{\ref{eqPlet2}}
\end{equation}
and otherwise $\pleth{\lambda}{\mu}{\nu} =0$.
\end{theorem}

\begin{proposition}\label{translation:pleth2}
Let $\lambda$, $\mu$ and $\nu$ be partitions. Let $n \geq 0$ and $k \in \ZZ$ be integers such that $\ell(\nu) \leq n$. 
Let $r$ and $q$ be defined as in Theorem \ref{symmetry:plethysm2}. 
Assume that $\lambda+(k^r)$ is a partition. If $\ell(\lambda) \leq r$ and $\nu+((qk)^n)$ is a partition, then
\[
\pleth{\lambda}{\mu}{\nu}=\pleth{\lambda+(k^r)}{\mu}{\nu+((qk)^n)},  \tag{\ref{eqPlet2:trans}}
\]
else $\pleth{\lambda}{\mu}{\nu}=0$.
\end{proposition}

\begin{remark}The number $r=\#\tab{\mu}{n}$ is given by the \emph{hook--content formula} (see \cite[pg. 376]{StanleyBookEC2}). 
\end{remark}

\begin{example}[Weight reduction for Plethysm  coefficients]\label{computational:plethysm}
As for Kronecker products (see example \ref{computational:kro}), plethysms are trivial in the basis of power sums. Plethysms coefficients can thus be computed by means of conversions to the power sums basis (this is done this way in SAGE and SF \cite{SAGE, SF}).  When performing such a computation, it is very helpful to reduce the weight of the symmetric functions involved. This can be done, in some cases, by means of Theorem \ref{symmetry:plethysm}. 
The weight for the plethysm coefficient $a_{\lambda,\mu}^{\nu}$ is $N:=|\nu|=|\lambda|\cdot |\mu|$. Theorem \ref{symmetry:plethysm} shows that this plethysm coefficient is equal to another plethysm coefficient with weight $\mu_1 \ell(\nu) |\lambda|-N$. We can also make use of the symmetries (\cite[I.\S8.Ex.1(a)]{MacdonaldBook})
\[
a_{\lambda,\mu}^{\nu}=
\left\lbrace
\begin{matrix}
a_{\lambda,\mu'}^{\nu'}  \text{ when $|\mu|$ is even,}\\
a_{\lambda',\mu'}^{\nu'} \text{ when $|\mu|$ is odd.}
\end{matrix}
\right.
\]
Set $K=\ell(\mu) \ell(\nu) |\lambda|$. We can obtain, therefore, a reduction to the weight 
\[
K \cdot \min\left(\frac{\mu_1}{\ell(\mu)},\frac{\nu_1}{\ell(\nu)}\right) - N
\]
\end{example}


\section{Kostka--Foulkes polynomials}\label{sec:Kostka}
In this section we use the definition of the Kostka-Foulkes polynomials given in  Equation \eqref{eq:stoP} and the method of previous sections to derive rectangular symmetries for them.

The specialization of Equation \eqref{eq:stoP} at a finite set of variables $X_n=\{x_1,x_2,\ldots,x_n\}$, with $n \geq \ell(\mu)$, of the Hall--Littlewood polynomial $P_{\mu}$ is given \cite[III.(2.1)]{MacdonaldBook} by
\begin{equation}\label{HL_def}
P_\mu(X_n;t) = \frac{1}{v_{\mu,n}(t)} \sum_{w\in S_n} w\left(x_1^{\mu_1}\cdots x_n^{\mu_n} \prod_{i<j} \frac{x_i-tx_j}{x_i-x_j}\right).
\end{equation}
where $w \in S_n$ permutes the variables $x_i$ and
\[
v_{\mu,n}(t)=\prod_{i}\prod_{r=1}^{m_i(\mu)}\frac{1-t^r}{1-t}
\]
with $m_i$ the number of occurrences of $i$ in the sequence $\mu$, once it has been padded with zeros to get length $n$. 

As in the case of Schur polynomials, this definition still makes sense perfectly when $\mu \in \P{n}$ (with possible negative coordinates). We get the following generalization of \eqref{schur_of_inverses}:
\begin{lemma} \label{HL-inverses}
Let $\mu$ be a weakly decreasing sequence of integers, of length $n$, and $X=\{x_1,x_2,\ldots,x_n\}$. We have
\begin{equation}\label{HL_of_inverses}
P_\mu(\inv{X_n}; t) =  {P_{\boxit{0}{n}{\mu}}(X_n;t)}
\end{equation}
and for any integer $k$, 
\begin{equation}\label{HL_facto}
P_{\mu+(k^n)}(X_n; t) =  (x_1 x_2 \cdots x_n)^k \cdot P_\mu(X_n;t).
\end{equation}
\end{lemma}
\begin{proof}
Specializing equation \eqref{HL_def} at $\inv{X}$, we get
\[
P_\mu(\inv{X_n}; t) = \frac{1}{v_{\mu,n}(t)} \sum_{w\in S_n} w\left(x_1^{-\mu_1}\cdots x_n^{-\mu_n} \prod_{i<j} \frac{x_i^{-1}-tx_j^{-1}}{x_i^{-1}-x_j^{-1}}\right).
\]
Observe that 
\[
\prod_{1\leq i<j\leq n} \frac{x_i^{-1}-tx_j^{-1}}{x_i^{-1}-x_j^{-1}} = \prod_{1\leq i<j\leq n} \frac{x_j - tx_i}{x_j-x_i}.\]
Let $w_0$ be the permutation that maps $i$ to $a-i+1$, i.e., the longest permutation. 
Then 
\[
 w_0\left(x_1^{-\mu_1}\cdots x_n^{-\mu_n} \prod_{1\leq i<j\leq n} \frac{x_j - tx_i}{x_j-x_i}\right) = x_1^{-\mu_n}\cdots x_n^{-\mu_1} \prod_{i<j} \frac{x_i-tx_j}{x_i-x_j}
 \]

For any sequence $\mu$ of length $n$, the number of times $i$ occurs in $\mu$ is equal to the number of times $-i$ occurs in $\boxit{0}{n}{\mu}$. Therefore the factor  
$v_{\mu,n}(t)$ is invariant under changing $\mu$ into $\boxit{0}{n}{\mu}$.
Equation \eqref{HL_of_inverses} follows now just by changing the order of summation. 

The proof of \eqref{HL_facto} is straightforward.
\end{proof}

Specializing equation \eqref{eq:stoP} at $\inv{X_n}$, we get
\[
s_\lambda[\inv{X_n}] = \sum_{\mu: \ell(\mu) \leq n} K_{\lambda,\mu}(t) P_\mu(\inv{X_n}; t).
\]
Using \eqref{schur_of_inverses} and Lemma \ref{HL-inverses}, we get
\[
s_{\boxit{0}{n}{\lambda}}[X_n]=\sum_{\mu\,:\, \ell(\mu) \leq n} K_{\lambda, \mu}(t) P_{\boxit{0}{n}{\mu}}({X_n};t). 
\]
Let $k$ be an integer. Let us multiply both sides with $(x_1 x_2 \ldots x_n)^k$. We get
\[
s_{\boxit{k}{n}{\lambda}}[X_n]=\sum_{\mu\,:\, \ell(\mu) \leq n} K_{\lambda, \mu}(t) P_{\boxit{k}{n}{\mu}}({X_n};t) 
\]
since it follows clearly from \eqref{HL_def} that $(x_1 x_2 \ldots x_n)^k P_{\nu}(X_n;t)=P_{\nu+(k^n)}(X_n;t)$ for any sequenece $\nu\in\P{n}$.

Assume now that $k \geq \lambda_1$. We see that if  $\mu \not\subset (k^n)$ then $K_{\lambda, \mu}(t)$ and else, 
\[
K_{\lambda, \mu}(t) = K_{\boxit{k}{n}{\lambda}, \boxit{k}{n}{\mu}}(t).
\]

We have proved the following result.
\begin{theorem} \label{symmetry:Kostka1var}
Let $k$ and $n$ be nonnegative integers. 

Let  $\lambda$ and $\mu$ be partitions such that $\lambda_1 \leq k$ and $\ell(\mu) \leq n$. 
If $\lambda \subseteq (k^n)$ and $\mu \subseteq (k^n)$ then
\[
K_{\lambda, \mu}(t) = K_{\boxit{k}{n}{\lambda}, \boxit{k}{n}{\mu}}(t).\tag{\ref{eq:Kostka}}
\]
Else $K_{\lambda, \mu}(t)=0$.
\end{theorem}

Starting again from \eqref{eq:stoP}, specializing it at $X_n$  and multiplying with $(x_1x_2\cdots x_n)^k$, and using \eqref{translation} and \eqref{HL_facto}, we get the following result.
\begin{proposition}\label{symmetry:Kostka1var:trans}
Let $n$ and $k$ be integers, with $n \geq 0$. Let $\lambda$ and $\mu$ be partitions, with $\ell(\mu) \leq n$ and such that $\lambda+(k^n)$ is a partition. 

If $\ell(\lambda) \leq n$ and $\mu+(k^n)$ is a partition then
\begin{equation}
K_{\lambda,\mu}(t) = K_{\lambda+(k^n),\mu+(k^n)}(t), \tag{\ref{eq:Kostka:trans}}
\end{equation}
else $K_{\lambda,\mu}(t)=0$.
\end{proposition}

\section*{Acknowledgments}  The authors want to thank J. Stembridge for sending us his preprints, and to R. King, F. Bergeron, and O. Azenhas for helpful conversations.
They also thank M. Zabrocki for his help in understanding how these symmetries generalize to Macdonald polynomials. The results of this investigation will appear as a separate note. 

R. Orellana is grateful for the hospitality of the University of Sevilla and IMUS. E. Briand and M. Rosas have been partially supported by projects MTM2010--19336, MTM2013-40455-P, FQM--333, P12--FQM--2696 and FEDER.  R. Orellana was partially supported by NSF Grant DMS-130512.

\bibliographystyle{amsplain}
\bibliography{AlphabetOfInverses.bib}




\end{document}